\documentclass[11pt]{article}
\usepackage{graphicx}
\usepackage[T1]{fontenc}
\usepackage[utf8]{inputenc}
\usepackage{amsmath, amsthm,latexsym,amsbsy,amssymb, color,mathtools}
\usepackage{enumerate}
\usepackage[normalem]{ulem}

\numberwithin{equation}{section}
\usepackage{graphicx}
\usepackage{ulem}
\font\bigbf=cmbx10 at 16pt

\def\ds{\displaystyle}
\def\forall{\hbox{for all}~}

\def\L{{\bf L}}

\def\ve{\varepsilon}
\def\thetam{\theta_{\min}}
\def\vmax{v_{\max}}

\def\D{{\mathcal D}}

\def\vs{\vskip 2em}

\def\v{\vskip 1em}

\def\Prob{\hbox{Prob.}}
\def\sqr#1#2{\vbox{\hrule height .#2pt
\hbox{\vrule width .#2pt height #1pt \kern #1pt
\vrule width .#2pt}\hrule height .#2pt }}

\def\bega{\begin{array}}\def\enda{\end{array}}
\def\begi{\begin{itemize}}\def\endi{\end{itemize}}

\def\Prob{\hbox{Prob.}}

\newcommand{\sg}[1]{\mathrm{sign}\big(#1\big)}

\newcommand{\set}[1]{\left\{#1\right\}}
\newcommand{\norm}[1]{\left\lVert#1\right\rVert}
\newcommand{\sett}[2]{\left\{#1 ~\Big| ~ #2\right\}}

\newcommand{\p}[1]{\left( #1 \right)}

\DeclareMathOperator*{\argmax}{arg\,max}
\DeclareMathOperator*{\argmin}{arg\,min}

\newtheorem{theorem}{Theorem}[section]
\newtheorem{proposition}{Proposition}[section]
\newtheorem{definition}[theorem]{Definition}
\newtheorem{lemma}[theorem]{Lemma}
\newtheorem{corollary}{Corollary}[section]

\begin{document}

\title{\bigbf
A model of debt with bankruptcy risk\\ and currency devaluation}
\vs
\author{\it Rossana Capuani, Steven Gilmore, and Khai T. Nguyen}
\author{Rossana Capuani, Steven Gilmore, and Khai T. Nguyen\\
\\
  {\small  Department of Mathematics, North Carolina State University. }\\
 \\  {\small e-mails: ~rcapuan@ncsu.edu, ~sjgilmo2@ncsu.edu,  ~ khai@math.ncsu.edu}
 }
\maketitle

\begin{abstract}
The paper studies  a system of Hamilton-Jacobi equations, arising from a stochastic optimal debt management problem in an infinite time horizon with exponential discount, modeled as a noncooperative interaction between a borrower and a pool of risk-neutral lenders. In this model, the borrower is a sovereign state that can decide how much to devaluate its currency and which fraction of its income should be used to repay the debt.  Moreover, the borrower has the possibility of going bankrupt at a random time and must declare bankruptcy if the debt reaches a threshold $x^*$. When bankruptcy occurs, the lenders only recover a fraction of their capital. To offset the possible loss of part of their investment, the lenders buy bonds at a discounted price which is  not given  a priori. This leads to a  nonstandard optimal control problem. We  establish an existence result of solutions to this system and in turn recover optimal feedback  payment strategy $u^*(x)$ and currency devaluation $v^*(x)$. In addition,  the behavior of $(u^*,v^*)$ near $0$ and $x^*$ is studied.
\end{abstract}
\section{Introduction}
Consider a system of  Hamilton-Jacobi equations
\begin{equation}\label{HJ0}
    \begin{cases}
        (r+\rho(x))V~ =~\ds \rho(x)\cdot B+H(x,V',p) + \dfrac{\sigma^2x^2}{2}\cdot V'' \,, \\[10pt]

        \p{ r + \lambda + v(x)}\cdot p - \p{ r + \lambda} ~=~ \rho(x)\cdot [ \theta(x) - p] + H_\xi(x, V', p) \cdot p' + \dfrac{(\sigma x )^2}{2}\cdot p'' \,,

    \end{cases}
\end{equation}
motivated by a stochastic optimal debt management problem in infinite time horizon with exponential discount, modeled as a noncooperative interaction between a borrower and a pool of risk-neutral lenders. Here, the independent variable $x$ is the debt-to-income ratio and $V$ is the value function for the
borrower who is a sovereign state that can decide the devaluation rate of its currency $v$  and the fraction of its income $u$ which is used to repay the debt. The borrower has a possibility to go bankrupt with an instantaneous bankruptcy risk $\rho$, and  must declare bankruptcy when the  debt-to-income ratio reaches a threshold $x^*>0$. The salvage function $\theta$ determines the fraction of capital that can be recovered by lenders when bankruptcy occurs. To offset the possible loss of part of their investment, the lenders buy bonds at a discounted rate $p$ which is  not given a priori.  Rather, it is determined by the expected evolution of the  debt-to-income ratio at all future times. Hence it depends globally on the entire feedback controls $u$ and $v$. This leads to a highly nonstandard optimal control problem, and  a ``solution'' must be understood as a Nash equilibrium, where the strategy implemented by the borrower represents the best reply to the strategy adopted by the lenders, and conversely.


In the economics literature,  some related models of debt and bankruptcy, with a focus on mathematical analysis, can be found in \cite{BJ,BMNP, BN1, AK, NT}. In \cite{BN1}, the borrower has a fixed income and large values of the debt determine a bankruptcy risk, adding uncertainty to the model. In \cite{NT}, a numerical analysis was performed of a similar model in which the uncertainty comes not from bankruptcy risk but the random evolution of the borrower's income. An analytical study of a slight variant of the model was performed in \cite{BMNP} where no currency devaluation is available to the government. The authors constructed optimal feedback solutions in the stochastic case and provided an explicit formula for the optimal strategy in the deterministic case. Moreover,  their analysis also shows how the expected total cost of servicing the total debt together with the bankruptcy cost  is affected by different choices of $x^*$. Interestingly, under a natural assumption, the possibility of a {\it ``Ponzi scheme''}, where the old debt is serviced by initiating more and more new loans, can be ruled out. This  study is continued by adding the option for currency devaluation in \cite{AK}. Recently,  a stochastic model with no currency devaluation but with uncertainty coming from bankruptcy risk was analyzed in \cite{BJ}.

The present paper aims to provide a detailed mathematical analysis of a more general model of debt and bankruptcy. When the currency devaluation is not an option for the borrower ($v \equiv 0$), or the bankruptcy risk
 \[
\rho(x) ~=~ \begin{cases}
	0 \quad & \mbox{ if } x < x^* \\
 	+\infty \quad & \mbox{ if } x = x^*,
 	\end{cases}
\]
our model reduces to the one analyzed in \cite{BJ} or \cite{AK}, respectively.  We  establish an existence result for the system of Hamilton-Jacobi equations (\ref{HJ0}).  In turn, this yields the existence of a pair of optimal feedback controls $(u^*,v^*)$ which minimizes the expected cost to the borrower. More precisely, let $L:[0,1[\to [0,+\infty[$ be the cost for the borrower to implement the control strategy $u(\cdot)$ and let $c:[0,v_{\max}[\to [0,+\infty[$ be a social cost resulting by devaluation, i.e., the increasing cost of the welfare and of the imported goods. We show that 
\begin{itemize}
\item{\it  given  the debt-to-income ratio threshold $x^*$ and the bankruptcy risk function $\rho$ with $\lim_{x\to x^*-}\rho(x)=+\infty$, if $L$ and $c$ are strictly convex then the system   (\ref{HJ0}) admits a $\mathcal{C}^2$ solution $(V,p)$ in $[0,x^*]$. This implies that an optimal feedback solution to  the model of debt and bankruptcy is
\[
u^*(x)~=~\underset{w\in [0,1]}{\mathrm{argmin}}\left\{ L(w) - w\cdot\dfrac{V(x)}{p(x)}\right\},\qquad v^*(x)~=~\underset{v\in [0,v_{\max}]}{\mathrm{argmin}}\left\{ c(v) - vxV'(x)]\right\}.
\]
}
\end{itemize}

Since  $\rho$ goes to $+\infty$ as $x$ tends to $x^*$, the system (\ref{HJ0}) is not uniformly elliptic at $0$ and $x^*$. To handle this difficulty,  the classical idea is to construct  solutions of approximate systems  as steady states of corresponding auxiliary  parabolic systems. In order to obtain a solution to the original system (\ref{HJ0}), we derive explicit a priori estimates on the derivatives of approximate solutions. As a consequence, the devaluation of currency is not optimal when the debt-to-income ratio $x$ is sufficiently small. In addition, we provide lower and upper bounds for the value function $V$ as subsolution and supersolution to the first equation of (\ref{HJ0}). Relying on these bounds, we show that when $x$ is sufficiently close to $x^*$
\begin{itemize}
{\it
\item if the risk of bankruptcy $\rho$ slowly approaches to infinity as $x$ tends to $x^*$, i.e.,
$$
\ds \int_{x^*-\delta}^{x^*}\rho(t)~dt<+\infty\qquad for~some~\delta>0,
$$
then the optimal strategy of borrower involves continuously devaluating its currency and making payment,
\item conversely, if the risk of bankruptcy $\rho$ quickly approaches to infinity infinity as $x$ tends to $x^*$, i.e.,
$$
 \lim_{x\to x^*-}\rho(x) (x^*-x)^2~=~+\infty,
$$
then any action to reduce the debt is not optimal.
}
\end{itemize}

The remainder of the paper is organized as follows. In Section 2, we provide a more detailed description of the model and derive the equations satisfied by the value function $V$ and the discounted bond price $p$. In Section 3, we begin by showing the existence of optimal controls in feedback form. We close by performing a more detailed analysis of the behavior of the optimal feedback controls near $x^*$, under certain assumptions on the bankruptcy risk.
\v
\section{Model derivation and system of Hamilton-Jacobi equations}
\setcounter{equation}{0}
\subsection{Model derivation}
We introduce  our optimal debt management problem in infinite time horizon, modeled as a noncooperative interaction between a borrower and a pool of risk-neutral lenders.  Let $v(t)$ be the devaluation rate at time $t$, regarded as a control. The total income $Y$ is governed by a stochastic process
\[
dY~=~(\mu +v(t)) Y(t) dt + \sigma Y(t) dW,
\]
where $W$ is a Brownian motion on a filtered probability space, $\sigma>0$ is the volatility, and $\mu$ is the average growth rate of the economy. We denote by $X(t)$ the total debt of a borrower, financed by issuing bonds, and
\begin{itemize}
\item $r$ = the interest rate paid on bonds;
\item $\lambda$ = the rate at which the borrower pays back the principal.
\end{itemize}
When an investor buys a bond of  unit nominal value, he will receive a continuous stream of payments with intensity $(r+\lambda) e^{-\lambda t}$. If no bankruptcy occurs, the payoff for an investor will thus be
\[
\int_0^\infty e^{-rt} (r+\lambda) e^{-\lambda t}\, dt~=~1.
\]
Otherwise, a lender recovers only a fraction $\theta\in [0,1]$ of his outstanding capital which depends on the total amount of debt at the time on bankruptcy. To offset this possible loss, the investor buys a bond with unit nominal value at a discounted price ${p}\in [0,1]$. Hence, as in \cite{NT}, the total debt evolves according to
\[
\dot X(t)~=~-\lambda X(t) + \dfrac{(\lambda + r) X(t) - U(t)}{p(t)},
\]
where $U(t)$ is the rate of payments that the borrower chooses to make to the lenders at time $t$. By It\=o's formula, one derives the stochastic evolution of the debt-to-income ratio $x = X/ Y$
\begin{equation}\label{SDE}
dx(t)~=~\left[ \left(\dfrac{\lambda+r}{p(t)} -\lambda + \sigma^2 -\mu - v(t)
\right) x(t) - \dfrac{u(t)}{p(t)}\right]\, dt - \sigma \,x(t)\, dW,
\end{equation}
where $u=U / Y \in [0,1]$ is the portion of the total income allocated to reduce the debt.

In this model, we assume there exists a threshold $x^*>0$ beyond which bankruptcy immediately occurs. Define $T_{x^*}$ as the time when the borrower's debt-to-income ratio first reaches $x^*$
\begin{equation}\label{TB}
T_{x^*}~:=~\inf\{ t>0\,;~~x(t)= x^*\}\,.
\end{equation}
As in \cite{BJ}, the borrower may go bankrupt at random time $\mathcal{T}_B$ before $T_{x^*}$. If at time $\tau$ the borrower is not yet bankrupt and the debt-to-income ratio is $x(\tau)=y$, then the probability that bankruptcy will occur shortly after time $\tau$ is measured by
\[
\text{Prob.}~\left\{\mathcal{T}_B\in [\tau,\tau+\ve]~\Big|~\mathcal{T}_B>\tau, x(\tau)=y\right\}~=~\rho(y)\cdot \ve+o(\ve)\,.
\]
Here, $\rho:[0,x^*[\to [0,+\infty[$ is the instantaneous bankruptcy risk, an increasing function of $x$ with $\lim_{x\to x^{*-}}\rho(x)=+\infty$. Hence, the probability that the borrower is not yet bankrupt at time $t>0$ is computed as
\begin{equation}\label{T-B}
\Prob~\{\mathcal{T}_B>t\}~=~\begin{cases} \exp\left\{\ds-\int_0^t\rho(x(\tau))~d\tau\right\}~~& \mathrm{if}~~~~t~<~T_{x^*},\cr\cr
0~~& \mathrm{if}~~~~t\geq T_{x^*}\,.\end{cases}
\end{equation}
Let $\theta(x)$ be the salvage rate that determines the fraction of the outstanding capital that can be recovered by lenders if bankruptcy occurs when the debt-to-income ratio has size $x$. As in \cite{BMNP, BN1, NT}, the discounted bond price is uniquely determined by the competition of a pool of risk-neutral lenders
\begin{multline}\label{P1}
p~=~ E\Big[\int_0^{\mathcal{T}_{B}}(r+\lambda)\exp\left\{-\int
_0^\tau \bigl(\lambda +r+v(x(s))\bigr)\,ds\right\}d\tau \\
    +\exp\left\{{\ds -\int^{\mathcal{T}_{B}}_{0}(r+\lambda+v(x(\tau))d\tau}\right\}\cdot \theta(x^*)\Big].
\end{multline}
Given an initial size $x_0$ of the debt-to-income ratio, the borrower wants to find a pair of optimal controls $(u,v)$ which minimizes his total expected cost, exponentially discounted in time:
\begin{equation}\label{C1}
    \mathrm{Minimize}~J[x_0,u,v]~=~E\left[\int_0^{\mathcal{T}_{B}} e^{-rt}\cdot L(u(t)) + c(v(t)) dt + e^{-r \mathcal{T}_{B}} B\right]_{x(0)=x_0}\,,
\end{equation}
where $c(v)$ is the social cost resulting from devaluation, $L(u)$ is the cost to the borrower for putting income towards paying the debt, 
and $B$ is the cost of bankruptcy.

To complete this subsection, we introduce standard assumptions in our model.  Concerning the functions $\theta,\rho,L$ and $c$, we shall assume
{\it
\begin{itemize}
\item [{\bf (A1)}] The map $\theta: [0,x^*] \to ]0,1]$ is non-increasing, continuous, and locally Lipschitz in $[0,x^*[$.
\item[\bf{(A2)}] The function $\rho: [0,x^*[\to [0,+\infty[$ is continuously differentiable for $x\in[0,x^*[$, and satisfies
    \begin{equation*}
    \rho(0)~=~0,\qquad \rho'(x)~\geq~ 0\qquad\mathrm{and}\qquad  \lim_{x \to x^*}\rho(x)~=~+\infty.
    \end{equation*}
    \item [{\bf (A3)}] The function $(L,c): [0,1[\times [0,\vmax[\to [0,+\infty[\times [0,+\infty[$ is twice continuously differentiable such that
    \[
    L'(u),~ c'(v)~>~0,\qquad L''(u),~c''(v)~\geq~\delta_0    \]
    and
    \[
    L(0)~=~c(0)~=~0,\qquad \lim_{u\to 1}L(u)~=~+\infty,\qquad \lim_{v\to \vmax}c(v)~=~+\infty
    \]
    for some constant $\delta_0 >0$ and $v_{\max}\geq 0$.

  \end{itemize}
}

\subsection{System of Hamilton-Jacobi equations} The stochastic control system (\ref{SDE})--(\ref{P1}) is not standard. Indeed, the discount price $p$ in (\ref{P1}) depends on the debt-to-income ratio not only at the present time $t$ but also at all future times. 
Therefore, it is natural to look at this model in a feedback form. More precisely, assume that the optimal control has feedback form, so that
\[
u~=~u^*(x),\qquad v~=~v^*(x)\qquad\qquad \hbox{for}\quad x\in [0,x^*].
\]
\begin{definition}[Stochastic optimal feedback solution]\label{OFC} We say that a triple of functions \\ $(u^*(x),v^*(x),p(x))$ provides an optimal solution to the problem of optimal debt management (\ref{SDE})--(\ref{C1}) if:
\begin{itemize}
\item [(i)] Given the function $p(\cdot)$,  for every initial value $x_0\in [0, x^*]$
the feedback control $(u^*,v^*)$ with bankruptcy time $\mathcal{T}_{B}$ as in (\ref{T-B})
provides an optimal solution to the
stochastic  control problem (\ref{C1}), with dynamics (\ref{SDE}).
\item [(ii)] Given the feedback control $(u^*(\cdot),v^*(\cdot))$, the discounted price $p$ satisfies (\ref{P1}), where $\mathcal{T}_B$ is the bankruptcy time (\ref{T-B}) determined by the dynamics
(\ref{SDE}).
\end{itemize}
\end{definition}
Let us introduce the associated Hamiltonian function to (\ref{SDE})-(\ref{P1})
\begin{equation}\label{H1}
    H(x,\xi,p)~=\min_{(u,v)\in [0,1[\times [0,\vmax[}\left\{L(u) + c(v) -\left(\frac{u}{p}+xv\right)\cdot \xi\right\}+ \left(\dfrac{\lambda+r}{p} -\lambda + \sigma^2 -\mu\right) x\, \xi
\end{equation}
and two functions
\begin{equation}\label{u-v-ti}
\tilde{u}(\xi,p)~=~\underset{u\in [0,1]}{\mathrm{argmin}}\left\{ L(u) - u\,\dfrac{\xi}{p}\right\},\qquad \tilde{v}(x,\xi)~=~\underset{v\in [0,\vmax] }{\mathrm{argmin}}\left\{ c(v) - vx\xi\right\}.\end{equation}
Under the assumption {\bf (A3)}, a direct computation yields
\begin{align}\label{uv-tilde}
\tilde{v}(x,\xi)=
    \begin{cases}
    0 \ \ \ \ &\mbox{if} \quad \xi x \leq c'(0)\cr\cr
    (c')^{-1}(\xi x) &\mbox{if} \quad \xi x > c'(0)
    \end{cases}
   \quad\mathrm{and}\quad
    \tilde{u}(\xi,p)=
    \begin{cases}
    0, \ \ \ \ &\mbox{if} \quad \frac{\xi}{p} \leq L'(0)\cr\cr
    (L')^{-1}\left(\frac{\xi}{p}\right), &\mbox{if} \quad  \frac{\xi}{p} > L'(0).
    \end{cases}
\end{align}
 Assume that the discount bond price, $p(\cdot)$ is given. The value function
 \[
  V(x_0) ~=~ \inf_{u,v} J[x_0,u,v],
\]
solves the following second-order ODE
\[
 (r+\rho(x))V(x)~= ~\ds \rho(x)\cdot B+H(x,V'(x),p(x)) + \dfrac{\sigma^2x^2}{2}\cdot V''(x),
\]
with boundary values
\[
V(0)~=~0\qquad\mathrm{and}\qquad V(x^*)~=~B.
\]
The optimality condition and (\ref{uv-tilde}) imply that the feedback strategies are recovered by
\begin{equation}\label{ufb}
    u^*(x) ~=~\tilde{u}(V'(x),p(x)) ~=~
        \begin{cases}
            0, &\text{ if }\quad  \dfrac{V'(x)}{p(x)}~ \leq~ L'(0), \\[10pt]
            (L')^{-1}\p{\dfrac{V'(x)}{p(x)}} ,&\text{ if }\quad \dfrac{V'(x)}{p(x)} ~>~ L'(0),
        \end{cases}
\end{equation}
and
\begin{equation}\label{vfb}
    v^*(x) ~=~ \tilde{v}(x,V'(x))~=~
        \begin{cases}
            0, &\text{ if }\quad  V'(x)x ~\leq ~ c'(0), \\[10pt]
            (c')^{-1}\p{ V'(x)x}  &\text{ if }\quad V'(x)x ~ > ~ c'(0) \,.
        \end{cases}
\end{equation}
On the other hand, suppose that a pair of optimal feedback controls $(u^*, v^*)$ is known. The Feynman-Kac formula {\cite{Shreve} implies that the discounted bond price $p$ is the solution to  second-order ODE
\begin{equation*}
    \begin{split}
        \p{ r + \lambda + v^*(x)}\cdot p(x) - \p{ r + \lambda} = \rho(x)  & \cdot [ \theta(x) - p(x)] + \\ & H_\xi(x, V'(x), p(x)) \cdot p'(x) + \dfrac{(\sigma x )^2}{2}\cdot p''(x),
    \end{split}
\end{equation*}
with $p(0)=1$ and $p(x^*)=\theta(x^*)$. Therefore,  finding  an optimal solution to the problem of optimal debt management (\ref{SDE})-(\ref{C1}) leads to  the following system of second order implicit ODEs
\begin{equation}\label{sys}
    \begin{cases}
        (r+\rho(x))V(x)~ =~\ds \rho(x)\cdot B+H(x,V'(x),p(x)) + \dfrac{\sigma^2x^2}{2}\cdot V''(x) \,, \\[10pt]
    \begin{aligned}
        \p{ r + \lambda + v(x)}\cdot p(x) - \p{ r + \lambda} ~=~ \rho(x)  & \cdot [ \theta(x) - p(x)] + \\ & H_\xi(x, V'(x), p(x)) \cdot p'(x) + \dfrac{(\sigma x )^2}{2}\cdot p''(x) \,,
    \end{aligned}\\
    v(x)~=~\underset{w\in [0,\vmax] }{\mathrm{argmin}} \set{ c(w) - wx V'(x)} \,,
    \end{cases}
\end{equation}
with boundary conditions
\begin{equation}\label{BC}
    V(0)~=~0, \quad V(x^*)~=~B \quad \text{and} \quad p(0)~=~1, \quad p(x^*)~=~\theta(x^*).
\end{equation}
To complete this subsection, let us collect some useful properties of Hamiltonian $H$.
\begin{lemma}\label{H-C1}
If $\textbf{(A3)}$ holds then $H$ is continuous differentiable and
its gradient at points $(x,\xi,p)\in [0,+\infty[\times[0,+\infty[\times]0,1]$ can be expressed by \begin{equation}\begin{cases}
\displaystyle H_{x}(x,\xi,p)~=& \displaystyle \Big[(\lambda+r)-p(\lambda +\mu +\tilde{v}(x,\xi)-\sigma^2)\Big]\cdot \frac{\xi}{p},\\
\displaystyle H_{\xi}(x,\xi,p)~=& \displaystyle \frac{1}{p}\cdot \Big[x\big((\lambda+r)-p(\lambda +\mu +\tilde{v}(x,\xi)-\sigma^2)\big)-\tilde{u}(\xi,p)\Big],\\
\displaystyle H_{p}(x,\xi,p)~=& \displaystyle (\tilde{u}(\xi,p)-x(\lambda +r))\cdot \frac{\xi}{p^2},
\end{cases}
\end{equation}
where the functions $\tilde{u},\tilde{v}$ are defined in (\ref{uv-tilde}). Furthermore,
 \begin{enumerate}
        \item [(i).] for all $\xi \in [0,+\infty[ $, the function $H$ satisfies
            \begin{equation*}
                \p{ \dfrac{ (\lambda +r)x -1}{p} + (\sigma^2 - \lambda - \mu - v(x))x} \cdot \xi~\leq~H(x, \xi, p)~\leq~\p{ \dfrac{\lambda + r}{p} - \lambda + \sigma^2 - \mu } x\xi \,,
            \end{equation*}
            \begin{equation*}
                \dfrac{ (\lambda +r)x -1}{p} + (\sigma^2 - \lambda - \mu - v(x))x~\leq~H_\xi(x, \xi, p) \leq \p{ \dfrac{\lambda + r}{p} - \lambda + \sigma^2 - \mu } x \,;
            \end{equation*}
        \item [(ii).] for every $(x,p) \in ]0,+\infty[\times ]0,+\infty[$ the map $\xi \mapsto H(x,\xi,p)$ is concave down and satisfies
            \begin{align}
                H(x,0,p) &~=~0 , \label{L3} \\
                H_\xi(x,0,p) &~=~\p{ \dfrac{\lambda + r}{p} - \lambda + \sigma^2 - \mu } x \,.\label{L4}
            \end{align}
    \end{enumerate}
\end{lemma}
\begin{proof} See \cite[Lemma B2 and Lemma B3]{AK}.
\end{proof}
\begin{corollary}\label{Hbound} Suppose the assumption ({\bf A3}) holds. Then for all $(x,\xi,p,v) \in [0,x^*]\times]0,+\infty[\times$ $[\thetam,+\infty[ \times [0,\vmax]$, it holds that
    $$ |H(x,\xi,p)|~ \leq~ K_1 \cdot \xi, \qquad \mbox{and} \qquad |H_\xi(x,\xi,p)|~ \leq~ K_1, $$
    where the constant $\theta_{\min} >0$ is defined in (\ref{theta-min}) and
   \begin{equation}\label{k1}
   K_1 ~:=~ \max\set{ \p{\dfrac{\lambda +r}{\thetam} + \sigma^2}x^* ,~ {\thetam}^{-1} + \p{ \lambda + \mu +  \vmax}x^*}.
   \end{equation}
\end{corollary}
\begin{proof}
    By Lemma \ref{H-C1} we compute an upper bound
    \begin{equation*}
        H(x,\xi,p) ~ \leq~ \p{ \dfrac{\lambda + r}{p} - \lambda + \sigma^2 - \mu}x\xi ~\leq~ \p{ \dfrac{\lambda + r}{\thetam} + \sigma^2 }x^* \xi ~\leq~ K_1 \cdot \xi
    \end{equation*}
    and lower bound
    \begin{align*}
        H(x,\xi,p) ~\geq&~
                \p{ \dfrac{ (\lambda +r)x -1}{p} + (\sigma^2 - \lambda - \mu - v(x))x} \cdot \xi  \\
         ~\geq&~
                -\p{ \dfrac{1}{\thetam} + (  \lambda + \mu + \vmax )x^*} \cdot \xi ~\geq ~ -K_1\cdot \xi,
    \end{align*}
    providing an uniform bound on $H$ with respect to $x$ and $p$. A similar analysis and application of Lemma \ref{H-C1} provides the uniform bound on $H_\xi$.
\end{proof}
\section{Optimal feedback solutions}
\setcounter{equation}{0}
In this section, we will construct a solution of the system of Hamilton-Jacobi equations (\ref{sys}) with  boundary conditions (\ref{BC}) for a given bankruptcy threshold $x^*$. In turn, this result yields the existence of optimal feedback controls for the problem of debt management (\ref{SDE})--(\ref{C1}). Finally, we show how the bankruptcy risk affects  the behavior of the optimal feedback control as the debt-to-income ratio tends to $x^*$.

\subsection{Existence results for system of Hamilton-Jacobi equations}
Before stating our main result, let us introduce the constant which will be a lower bound of the discount bond price $p$,
\begin{equation}\label{theta-min}
\thetam~:=~\min\left\{\theta(x^*), \dfrac{r+\lambda}{r+\lambda+\vmax}\right\}.
\end{equation}
\begin{theorem}\label{ext}  Under the standard assumptions  {\bf (A1) - (A3)}, the system of second order ODEs \eqref{sys} with boundary conditions \eqref{BC} admits a solution $(V,p):[0,x^*]\rightarrow [0,B]\times \big[\thetam,1\big]$ of class $C^2$ such that $V$ is monotone increasing and
\begin{equation}\label{vas}
    v(x)~=~\argmin_{w\geq 0}\{c(w)-wxV' \}~=~0  \qquad \forall \quad x\in\left[0, \dfrac{c'(0)}{M^*} \right].
\end{equation}
for a constant $M^*$ which can be explicitly computed .
\end{theorem}
It is well-known (see \cite[Theorem 4.1]{FR} or \cite[Theorem 11.2.2]{O}) that having constructed a solution $(V,p)$ to the boundary value problem  (\ref{sys})-(\ref{BC}), then $(u^*,v^*)$ in (\ref{ufb})-(\ref{vfb}) is an optimal solution to the problem of optimal debt management (\ref{SDE})-(\ref{C1}).
As a consequence of Theorem \ref{ext}, we obtain the following result.

\begin{corollary}
Under the same assumptions of Theorem \ref{ext}, the debt management problem (\ref{SDE})-(\ref{C1}) admits an optimal control strategy in feedback form. Moreover, there exists a threshold such that the optimal control strategy does not use currency devaluation for values  below that threshold.
\end{corollary}
Toward the proof of Theorem \ref{ext}, we introduce a system of second order implicit ODEs which approximates (\ref{sys}). More precisely, for any given $\ve>0$, let $\rho_\ve$ be  a monotone increasing  Lipschitz function on $[0,x^*]$ defined by
\begin{equation}\label{rho-ve}
    \rho_{\ve}(x)~=~
    \begin{cases}
    \rho(x) \ \ \ &\mbox{if}\qquad x\in[0, x_{\ve}]\\[1.8ex]
   \ds \frac{1}{\ve} &\mbox{if}\qquad  x\in[x_{\ve}, x^*]
    \end{cases}
 \qquad\mathrm{with}\qquad x_{\ve}~:=~ \rho^{-1}\left(\frac{1}{\ve}\right).
\end{equation}
Consider the following system of implicit ODEs
\begin{equation}\label{sysel}
    \begin{cases}
    (r+\rho_\ve(x))\cdot V~= ~\rho_\ve(x)\cdot B+H(x,V',p+\ve) + \Big(\dfrac{\sigma^2x^2}{2}+\ve\Big)\cdot V'' \,, \\[1.5ex]
    \begin{aligned}
    (r + \lambda + \tilde{v}(x,V'))\cdot p- (r + \lambda)~&=~ \rho_\ve(x)   \cdot [ \theta(x) - p] \\[0.8ex]
    &+ H_\xi(x, V', p+\ve) \cdot p' + \p{\dfrac{(\sigma x )^2}{2}+\ve } \cdot p'' \,,
    \end{aligned}
    \end{cases}
\end{equation}
where the function
\begin{align}\label{vtilde}
\tilde{v}(x,\xi)~=~
    \begin{cases}
    0, \ \ \ \ &\mbox{if} \qquad \xi x \leq c'(0),\\[1.2ex]
    (c')^{-1}(\xi x), &\mbox{if} \qquad \xi x > c'(0).
    \end{cases}
\end{align}
From the assumption {\bf (A3)}, one can easily get that
\begin{equation}\label{cond-v}
\tilde{v}(x,\xi)~\leq~\vmax,\qquad |\tilde{v}_x(x,\xi)|~\leq~\frac{1}{\delta_0}\cdot |\xi|\qquad \mbox{and}\qquad |\tilde{v}_\xi(x,\xi)|~\leq~\frac{1}{\delta_0}\cdot |x|.
\end{equation}
We  establish an existence result for $(\ref{sysel})$ with boundary condition (\ref{BC}) by considering the auxiliary  parabolic system whose
steady states will provide a solution to $(\ref{sysel})$. Following \cite{Amann}, we shall construct a compact, convex and positively invariant
set of functions $(V,p): [0, x^*]^2\mapsto [0,B]\times [\thetam,1]$. A topological technique will then yield the existence of solutions $(V_{\ve},p_{\ve})$ of the system (\ref{sysel}).
\begin{lemma}\label{Ex-app} Assume that  {\bf (A1) - (A3)}  hold. Then the system of ODEs (\ref{sysel}) with boundary conditions (\ref{BC}) admits a $C^2$ solution $(V_{\ve},p_{\ve}): [0,x^*]^2\to [0,B]\times [\thetam,1]$ such that $V_{\ve}$ is increasing.
\end{lemma}
\begin{proof} Given a threshold of bankruptcy $x^*>0$, let's consider the parabolic system with the unknown ${\bf V}(t,x)$ and ${\bf p}(t,x)$
\begin{equation}\label{syspa}
    \begin{cases}
    {\bf V}_t~=~-(r+\rho_\ve(x)){\bf V}+ ~\ds \rho_\ve(x)\cdot B+H(x,{\bf V}_x,{\bf p}+\ve) + \Big(\dfrac{\sigma^2x^2}{2}+\ve\Big) {\bf V}_{xx}\,, \\[10pt]
    \begin{aligned}
    {\bf p}_t~=~-(r + \lambda + \widetilde{v}(x,{\bf V}_x))\cdot {\bf p} &+ (r + \lambda) + \rho_\ve(x)   \cdot [ \theta(x) - {\bf p}] + \\ & H_\xi(x, {\bf V}_x, {\bf p}+\ve) \cdot {\bf p}_x+ \Big(\dfrac{(\sigma x )^2}{2}+\ve\Big)\cdot {\bf p}_{xx} \,,
    \end{aligned} 
    \end{cases}
\end{equation}
with the boundary conditions
\begin{equation}\label{BC1}
\begin{cases}
{\bf V}(t,0)~=~0 \\[1.5ex]
{\bf V}(t,x^*)~=~B
\end{cases}
\qquad\mathrm{and}~~\qquad
\begin{cases}
{\bf p}(t,0)~=~1\\[1.5ex]
{\bf p}(t,x^*)~=~\theta(x^*).
\end{cases}
\end{equation}
It is well-known (see \cite[Theorem 1]{Amann}) that the parabolic system (\ref{syspa}) with initial data
\begin{equation}\label{id}
{\bf V}(0,x)~=~V_0(x) \qquad\mbox{and} \qquad {\bf p}(0,x)~=~p_0(x),
\end{equation}
admits a unique solution $({\bf V}(t,x),{\bf p}(t,x))\in C^2([0,T]\times [0,x^*])\times C^2([0,T]\times [0,x^*])$ for any $T>0$. Adopting a semigroup notation, we denote by $S_t(V_0,p_0)=({\bf V}(t,\cdot),{\bf p}(t,\cdot))$ the solution to the system \eqref{syspa} at time $t$ with initial data \eqref{id}.
\medskip \\
We claim that the following closed and convex domain in $C^2([0,x^*])\times C^2([0,x^*])$
\begin{equation*}
\D~=~\Big\{ (V,p):[0,x^*]^2 \rightarrow [0,B]\times [\thetam,1]: (V,p)\in C^2, V' \geq 0,~ \mbox{and}~\eqref{BC}~\mbox{holds}\Big\},
\end{equation*}
is positively invariant under the semigroup $S_t$, namely
\begin{equation*}
S_t(\D)~\subseteq~\D \qquad \forall ~ t \geq 0 \,.
\end{equation*}
Indeed, as in \cite{BMNP}, we consider the constant functions $({\bf V}^{\pm},{\bf p}^{\pm})$ defined on $[0,\infty[\times [0,x^*]$ such that
\[
({\bf V}^+,{\bf p}^+)~\equiv~(B,1)\qquad\mathrm{and}\qquad ({\bf V}^-,{\bf p}^-)~\equiv~(0,\thetam)\,.
\]
Recalling \eqref{L3}, one has
\begin{align*}
-(r+\rho_\ve(x)){\bf V}^+ + \rho_\ve(x)\cdot B+ H(x,{\bf V}_x^+,{\bf p}+\ve) + \Big(\dfrac{\sigma^2x^2}{2}+\ve\Big) {\bf V}_{xx}^+ \\[1.5ex]
~=~-(r+\rho_\ve(x)){\bf V}^++ \rho_\ve(x)\cdot B~=~-rB~\leq~0
\end{align*}
and
\[
-(r+\rho_\ve(x)){\bf V}^- + \rho_\ve(x)\cdot B+ H(x,{\bf V}_x^-,{\bf p}+\ve) + \Big(\dfrac{\sigma^2x^2}{2}+\ve\Big) {\bf V}_{xx}^-~=~ \rho_\ve(x)\cdot B~\geq~0.
\]
This implies that ${\bf V}^+$ is a supersolution and ${\bf V}^-$ is a subsolution of the first parabolic equation in \eqref{syspa}. A  standard comparison principle \cite{WW} yields
\[
    0~=~{\bf V}^-(t,x)~\leq~{\bf V}(t,x)~\leq~{\bf V}^+(t,x)~=~B\qquad\forall (t,x)\in [0,+\infty[ \times [0,x^*].
\]
Similarly, from ({\bf A1})-({\bf A3}), (\ref{theta-min}) and (\ref{vtilde}), it holds
\begin{align*}
-(r + \lambda + \widetilde{v}(x,{\bf V}_x))\cdot {\bf p^+} + (r + \lambda) + \rho_\ve(x)   \cdot [ \theta(x) - {\bf p}^+] +  H_\xi(x, {\bf V}_x, {\bf p}^++\ve) \cdot {\bf p}^+_x\\[1.5ex]
+ \Big(\dfrac{(\sigma x )^2}{2}+\ve\Big)\cdot {\bf p}^+_{xx}~=~ - \tilde{v}(x,{\bf V}_x) +\rho_\ve(x)   \cdot [ \theta(x) -1]~\leq~0
\end{align*}
and
\begin{multline*}
-(r + \lambda + \widetilde{v}(x,{\bf V}_x))\cdot {\bf p^-} + (r + \lambda) + \rho_\ve(x)   \cdot [ \theta(x) - {\bf p}^-] +  H_\xi(x, {\bf V}_x, {\bf p}^-+\ve) \cdot {\bf p}^-_x\\+ \Big(\dfrac{(\sigma x )^2}{2}+\ve\Big)\cdot {\bf p}^-_{xx}~=~-(r + \lambda + \widetilde{v}(x,{\bf V}_x))\cdot\thetam+(r + \lambda) \\+ \rho_\ve(x)   \cdot [ \theta(x) -\thetam]
~\geq~(r+\lambda)-(r+\lambda+\vmax)\cdot \thetam~\geq~0.
\end{multline*}
Thus, ${\bf p}^+$ is a supersolution and ${\bf p}^-$ is a subsolution of the second parabolic equation in \eqref{syspa}, and this yields
\[
\thetam~=~{\bf p}^-(t,x)~\leq~{\bf p}(t,x)~\leq~{\bf p}^+(t,x)~=~1\qquad\forall (t,x)\in [0,+\infty[\times [0,x^*].
\]
Setting ${\bf W}(t,x) := {\bf V}_{x}(t,x)$ with initial condition  ${\bf V}(0,x) = V_0(x) \in \mathcal D$, we have
\begin{equation*}
\lim_{x \to 0^+}{\bf W}(t,x) \geq 0, \quad \mbox{and} \quad \lim_{x \to x^{*-}}{\bf W}(t,x) \geq 0, \quad \forall t \in ]0, + \infty [ \,,
\end{equation*}
and, by definition of $\mathcal D$,
\[
 {\bf W}(0,x) ~\geq~ 0, \qquad \forall x \in [0,x^*].
\]
Differentiating the first equation in \eqref{sysel}, we obtain that $\bf{W}$ solves the following ODE
\begin{equation}\label{W}
  {\bf W}_t ~=~ -(r + \rho_\ve){\bf W} + \rho_\ve'(B - {\bf V } ) + H_x + H_\xi {\bf W}_{x} + H_p {\bf p}_{x} + \sigma^2x {\bf W}_{x} + \p{\dfrac{\sigma^2x^2}{2} + \ve}{\bf W}_{xx}.
\end{equation}
Since
\[
    (H_x,H_{p})(x,0,{\bf p}+\ve)~=~0 \qquad \mathrm{and} \qquad {\bf V}(t,x)~\leq~B\qquad\forall(t,x)\in [0,+\infty[\times [0,x^*],
\]
one can easily see that the constant function $0$ is a subsolution to \eqref{W}. Thus, by \cite{WW},
\[
{\bf W}(t,x)~\geq~0\qquad\forall(t,x)\in [0,+\infty[\times [0,x^*]\,,
\]
yielding the monotone increasing property of ${\bf V}$ with respect to $x$.
From the bounds in Lemma \ref{H-C1} and the invariance of $\D$, we can apply \cite[Theorem 3]{Amann} to obtain the existence of a steady state solution $(V_\ve, p_\ve) \in \D$ for the system \eqref{syspa} which solves (\ref{sysel}) and  (\ref{BC}) such that $V_\ve$ is monotone increasing.
\end{proof}
In order to obtain a solution to (\ref{sys}) from  (\ref{sysel}) by passing $\ve$ to $0+$, we  derive some a priori estimates on the derivatives of $(V_{\ve},p_{\ve})$.

\begin{lemma}\label{p-b} Under the same assumptions in Lemma \ref{Ex-app}, let $(V_\ve,p_\ve)$ be a solution to \eqref{sysel} and \eqref{BC}. Then for every $0<\ve<1/2$, it holds
\begin{equation}\label{b-V'}
\|V'_{\ve}\|_{{\bf L}^{\infty}([0,x^*])}~\leq~M^*\qquad\mathrm{and}\qquad v_{\ve}(x)~=~0\qquad\forall x\in \left[0,{c'(0)\over M^*}\right],
\end{equation}
where the constant $M^*$ is explicitly computed by
$$
M^*~:=~\max\left\{M_1,\exp\left({3K_1x^*\over 2\sigma^2x_1^2}\right)\cdot \left({4B\over x^*}+{B\over K_1}\cdot {\rho\left(3x^*\over 4\right)}\right), \exp\left({2K_1x^*\over \sigma^2 x_1^2}\right)\cdot\left({4B\over x^*}+{rB\over K_1}\right)\right\},
$$
with $K_1$ defined in (\ref{k1}), and
\begin{equation}\label{conts}
x_1~:=~\min\set{ \dfrac{1}{6(\lambda + r + \sigma^2)}, \dfrac{x^*}{2}},\qquad M_1~:=~8\p{ L\p{\frac{1}{2}} + \rho(x_1)\cdot B}.
\end{equation}
Moreover, for any $\delta\in ]0,x^*/2[$, there exists a constant $M_{\delta}>0$ such that
\begin{equation}\label{bbb}
\|V''_{\ve}\|_{{\bf L}^{\infty}(]\delta,x^*-\delta[)}+\|p'_{\ve}\|_{{\bf L}^{\infty}(]\delta,x^*-\delta[)}+\|p''_{\ve}\|_{{\bf L}^{\infty}(]\delta,x^*-\delta[)}~\leq~M_{\delta}.
\end{equation}
\end{lemma}
\begin{proof} {\bf 1.} Let us first prove (\ref{b-V'}). Let $x_1$ and $M_1$ be as in \eqref{conts}. From (\ref{H1}) and Lemma \ref{Ex-app}, for every $(x,\xi)\in [0,x_1]\times [M_1,\infty[$, it holds
\begin{eqnarray*}
 H(x,\xi,p_{\ve}+\ve)&\leq&\min_{u \in[0,1]} \set{ L(u) - \dfrac{u}{p_{\ve}+\ve} \xi } + \p{ \dfrac{\lambda +r}{p_{\ve} + \ve} + \sigma^2 }x\xi\cr\cr
 &\leq& L\left({1\over 2}\right)+\dfrac{3(\lambda +r + \sigma^2)x - 1 }{2(p_{\ve} + \ve)} \cdot \xi\cr\cr
 &\leq&L\left({1\over 2}\right)-{\xi\over 8}~\leq~L\left({1\over 2}\right)-{M_1\over 8}\,.
 \end{eqnarray*}
 From the definition of $\rho_{\ve}$ in (\ref{rho-ve}), one has
 \begin{equation}\label{p-pve}
 \rho_{\ve}(x)~\leq~\rho(x)\qquad\forall x\in [0,x^*],
 \end{equation}
 and  {\bf (A2)} implies that
 \[
 \rho_{\ve}(x)B+ H(x,\xi,p_{\ve}+\ve)~\leq~\rho(x_1)B+L\left({1\over 2}\right)-{M_1\over 8}~<~0,
 \]
 for all $(x,\xi)\in [0,x_1]\times [M_1,\infty[$. Since $V_{\ve}$ is non-negative, the first equation of (\ref{sysel}) yields
 \[
 V_{\ve}''(x)~>~0\qquad\forall x\in ]0,x_1],
 \]
 provided that $V_{\ve}'(x)\geq M_1$ for all $x\in [0,x_1]$. Recalling from Lemma \ref{Ex-app} that $V_{\ve}'$ is non-negative, we have that   $\|V'_{\ve}\|_{{\bf L}^{\infty}([0,x_1]) }$ is bounded by $M_1$ or the maximal of $V_{\ve}'$ in $[0,x^*]$ is obtained at
 \[
x_m~:=~\argmax_{x \in [0,x^*]} V'_\ve(x)~>~x_1\,.
 \]
Let us establish an upper bound of $V'_{\ve}$ in $[x_1,x^*[$. Since $0\leq x_1\leq {x^*\over 2}$, by the mean value theorem, there exists a point $x_2 \in  ]x_1,{3\over 4}{x^*}[$ such that
\begin{equation}\label{mvt}
    V_\ve'(x_2)~=~ \dfrac{V_\ve(\frac{3}{4} x^*) - V_\ve(x_1)}{\frac{3}{4}x^*- x_1}~\leq~{B\over {3\over 4}x^*-{1\over 2}x^*}~\leq~\dfrac{4B}{x^*} \,.
\end{equation}
From the first equation of \eqref{sysel} we have
\begin{equation}\label{V''-1}
V''_\ve(x)~=~\dfrac{2}{\sigma^2  x^2 + 2 \ve}\cdot \left[ rV_\ve(x) + \rho_\ve(x) \p{ V_\ve(x) - B}   - H(x,V'_\ve(x),p_\ve(x) + \ve)\right].
\end{equation}
Two cases are considered:
\begin{itemize}
\item For any $x \in [x_1,x_2]$, from (\ref{p-pve}), {\bf (A2)}, and the above equality we have
\begin{align*}
    V''_\ve(x) &~\geq~
    \dfrac{-2}{\sigma^2  x^2 + 2 \ve}\cdot  \p{\rho_\ve(x)  B+ \left| H(x,V'_\ve(x),p_\ve(x) + \ve) \right| }\\
    &~\geq~{-2\over \sigma^2x_1^2}\cdot \p{\rho\left({3x^*\over 4}\right)  B+ \left| H(x,V'_\ve(x),p_\ve(x) + \ve) \right| },
\end{align*}
where the last inequality holds because $x_2\leq {3x^*\over 4}$. Since $p_{\ve}(x)\geq \theta_{\min}$, by Lemma \ref{Hbound} it holds that
\begin{equation*}
   V''_\ve(x) 
     ~\geq~ \dfrac{-2}{\sigma^2  x_1^2 } \p{ \rho\left({\frac{3x^*}{ 4}}\right)  B     + K_1 \cdot V'_\ve(x) },
     \end{equation*}
where the constant $K_1$ is defined in (\ref{k1}). Thus, applying Gr\"onwall's inequality in the interval $[x,x_2]$ with $x\in [x_1,x_2]$, we get
\begin{equation*}
V'_{\ve}(x)~\leq~  \left(V_{\ve}'(x_2)+{B\over K_1}\cdot {\rho\left(3x^*\over 4\right)}\right) \cdot \exp \left( {2K_1\over \sigma^2x_1^2}(x_2-x) \right) -{B\over K_1}\cdot {\rho\left(3x^*\over 4\right)}.
\end{equation*}
Recalling (\ref{mvt}) we obtain that
\begin{equation}\label{b-x1-x2}
\|V'_{\ve}\|_{{\bf L}^{\infty}([x_1,x_2])}~\leq~ \left({4B\over x^*}+{B\over K_1}\cdot {\rho\left(3x^*\over 4\right)}\right) \cdot \exp\left({3K_1x^*\over 2\sigma^2x_1^2}\right).
\end{equation}
\item Similarly, for any $x\in [x_2,x^*]$, it holds
\begin{eqnarray*}
V''_{\ve}(x)&\leq&\dfrac{2}{\sigma^2  x_2^2 }  \p{ rB +  K_1 \cdot V'_\ve(x)}~\leq~ \dfrac{2}{\sigma^2  x_1^2 }  \p{ rB +  K_1 \cdot V'_\ve(x)}
\end{eqnarray*}
and Gr\"onwall's inequality implies that
\begin{equation*}
V'_{\ve}(x)~\leq~ \left(V'_{\ve}(x_2)+{rB\over K_1}\right) \cdot \exp\left({2K_1\over \sigma^2 x_1^2}(x-x_2)\right)-{rB\over K_1}.
\end{equation*}
Thus, (\ref{mvt}) yields
\begin{equation}\label{b-x2-x*}
\|V'_{\ve}\|_{{\bf L}^{\infty}([x_2,x^*])}~\leq~ \left({4B\over x^*}+{rB\over K_1}\right)\cdot \exp\left({2K_1x^*\over \sigma^2 x_1^2}\right).
\end{equation}
\end{itemize}
Therefore, combining (\ref{conts}), (\ref{b-x1-x2}), (\ref{b-x2-x*}) and (\ref{vtilde}), we obtain (\ref{b-V'}).
\v
{\bf 2.} For any fixed $0<\delta<{x^*\over 2}$, we will provide uniform bounds on $\norm{V''_\ve}_{\L^\infty([\delta,x^* - \delta])}$, $\norm{p'_\ve}_{\L^\infty([\delta,x^* - \delta])}$, and  $\norm{p''_\ve}_{\L^\infty([\delta,x^* - \delta])}$. From (\ref{V''-1}), (\ref{p-pve}) and Lemma  \ref{Hbound}, it holds
\begin{eqnarray*}
|V_{\ve}''(x)|&\leq&{2\over \sigma^2\delta^2}\cdot\left[(r+\rho(x^*-\delta))\cdot B+\big|H(x,V'_\ve(x),p_\ve(x) + \ve)\big|\right]\cr\cr
&\leq&{2\over \sigma^2\delta^2}\cdot\left[(r+\rho(x^*-\delta))\cdot B+K_1\cdot V'_{\ve}(x)\right] ,
\end{eqnarray*}
for all $x\in [\delta,x^*-\delta]$. Thus, (\ref{b-V'}) implies that
\begin{equation}\label{b-V''}
\|V_{\ve}''\|_{{\bf L}^{\infty}([\delta,x^*-\delta])}~\leq~{2\over \sigma^2\delta^2}\cdot\left[(r+\rho(x^*-\delta))\cdot B+K_1M^*\right].
\end{equation}
Similarly, from the second equation in (\ref{sysel}), it holds that
\begin{equation}\label{p''-1}
|p_{\ve}''(x)|~\leq~{2\over \sigma^2\delta^2}\cdot\left[K_1\cdot |p'_{\ve}(x)|+ (r+\lambda+v_{\max}+\rho(x^*-\delta))\right]\quad\forall x\in [\delta,x^*-\delta].
\end{equation}
By the mean value theorem, there exists a point $x_3 \in  ]\delta,x^*-\delta[$ such that
\[
\left|p'_{\ve}(x_3)\right|~=~\left|{p_{\ve}(x^*-\delta)-p_{\ve}(\delta)\over x^*-2\delta}\right|~\leq~{1-\theta_{\min}\over x^*-2\delta}.
\]
Applying Gr\"onwall's inequality to (\ref{p''-1}) in the intervals $[\delta,x_3]$ and $[x_3,x^*-\delta]$, yields
\[
|p'_{\ve}(x)|~\leq~K_{\delta}\qquad\forall x\in [\delta,x^*-\delta],
\]
for some constant $K_{\delta}$ depending only on $\delta$. Thus, from \eqref{p''-1} we get
\[
|p''_{\ve}(x)|~\leq~{2\over \sigma^2\delta^2}\cdot\left[K_1K_{\delta}+ (r+\lambda+v_{\max}+\rho(x^*-\delta))\right]\quad\forall x\in [\delta,x^*-\delta].
\]
Combining the two above estimates and (\ref{b-V''}), we obtain (\ref{bbb}) with the constant
\[
M_{\delta}~:=~{2\over \sigma^2\delta^2}\cdot\left[(r+\rho(x^*-\delta))\cdot B+K_1M^*+K_1K_{\delta}+ (r+\lambda+v_{\max}+\rho(x^*-\delta))\right]+K_{\delta} \,.
\]
This completes the proof.
\end{proof}
We are ready  to prove our first main theorem.
\begin{proof}[{\bf Proof of Theorem \ref{ext}}] The proof is divided into three steps:
\smallskip \\
{\bf Step 1.} For any $0<\ve<1/2$ sufficiently small, let $(V_{\ve},p_{\ve})$ be a solution to \eqref{sysel} and \eqref{BC} which is constructed in Lemma \ref{Ex-app}.
Recalling Lemma \ref{H-C1} and \eqref{b-V'}, we obtain that $H$ and $H_\xi$ are Lipschitz continuous on $[\delta, x^* - \delta] \times [0,M^*] \times [\thetam,1]$ for any $\delta \in ]0,x^*/2[$. Using the a priori estimates (\ref{b-V'}) and (\ref{bbb}) in Lemma \ref{p-b}, and assumptions {\bf (A1)-(A2)}, the system (\ref{sysel})  implies that  the functions $V_\ve''$ and $p_\ve''$ are also uniformly Lipschitz on $[\delta,x^*-\delta]$. Thus, one can apply the Ascoli-Arzel\`a Theorem to extract a subsequence $(V_{\ve_n},p_{\ve_n})_{n\geq 0}$ with $\lim_{n\to \infty}\ve_{n}=0$ such that
$(V_{\varepsilon_n}, p_{\varepsilon_n})$ converges uniformly to $(V,p)$ in $C^2(]\delta,x^*-\delta[)$ for all $\delta>0$, where $V,p$ are twice continuously differentiable and solve the system of ODEs \eqref{sysel} on the open interval $]0, x^*[$. Moreover, since $V'_{\ve}$ is uniformly bounded by $M^*$ on $[0,x^*]$,
\[
\lim_{n\to\infty}~\|V_{\ve_n}-V\|_{{\bf L}^{\infty}([0,x^*])}~=~0 \,,
\]
which implies that
\[
V(0)~=~\lim_{n\to\infty}~V_{\ve_n}(0)~=~0,\qquad V(x^*)~=~\lim_{n\to\infty}~V_{\ve_n}(x^*)~=~B.
\]
{\bf Step 2.} It remains to check the boundary condition (\ref{BC}) for $p$.  Let us first show that
  \begin{equation}
        \label{px0} \lim_{x \to 0^+} p(x) ~=~ 1  \,.
\end{equation}
Given $\ve \in ]0,{1 \over 2}[$, we construct a lower bound $p^-$ of $p_\ve$ independent of $\ve$ in a neighborhood of $0$. From ({\bf A1}), there exists $M >0$ such that
%
     \begin{equation}\label{theta-l}
     	\theta(x) ~\geq~ 1 - Mx \quad \forall x \in [0,x^*/2].
     \end{equation}
Set $\bar{x}_0:= \min\set{\dfrac{c'(0)}{M^*}, \dfrac{x^*}{2},\ds{1-\theta_{\min}\over M}}$ where $M^*$ is the constant in \eqref{b-V'}.
%
%
%
%
   The sub-solution candidate is
    \begin{equation*}
        p^-(x) ~=~ 1 - kx^\gamma,  \,
    \end{equation*}
    with
    \[
     \gamma ~:=~ \min\set{1, (r + \lambda) \p{ \dfrac{\lambda+r}{\theta_{\min}} + \sigma^2 }^{-1}}\qquad \mbox{and} \qquad k ~:=~\left(1-\thetam\right) \cdot \bar x_0^{-\gamma}.
     \]
     Note that $p^-(0) = 1 \leq p_\ve(0) $. By choice of $k$ and $\bar{x}_0$, we have
     \[
     p^-(\bar x_0) ~=~ 1 - k\bar x_0^{\gamma} ~=~ 1 - \left(1-\thetam \right) \cdot \bar x_0^{-\gamma} \cdot \bar x_0^\gamma ~=~ \thetam
     \]
and from (\ref{theta-l}), it holds for all $x \in [0,\bar x_0]$ that
     \begin{equation}\label{ttt}
    	p^-(x) ~=~ 1 - \left(1-\thetam\right) \cdot \bar x_0^{-\gamma}\cdot x^\gamma ~\leq~ 1 - M \bar x_0^{1-\gamma}\cdot x^\gamma ~\leq~1-Mx~\leq~\theta(x).
     \end{equation}
We claim that $p^-$ is a sub-solution of the second equation in \eqref{sysel} in the interval $[0,\bar{x}_0]$. Indeed, recalling \eqref{b-V'} that  $v_\ve(x) \equiv 0$ on the region $x \in [0,\bar{x}_0]$,  from Lemma \ref{H-C1} (i) and (\ref{ttt}), one estimates
    \begin{align*}
        (r + \lambda) -& (r+ \lambda +v_\ve(x)) p^- + \rho_\ve(x) [ \theta(x) - p^-(x)] + H_\xi(x,V,p^-)\cdot (p^-)' + \p{ \ve + \dfrac{\sigma^2 x^2}{2}} (p^-)'' \\
        &\geq~ (r + \lambda)kx^\gamma - \gamma k x^{\gamma -1} H_\xi(x,V,p^-) + \p{ \ve + \dfrac{\sigma^2 x^2}{2}} \gamma (1 - \gamma) kx^{\gamma -2} \\
               &\geq~ (r + \lambda)kx^\gamma  - \gamma k x^{\gamma -1} H_\xi(x,V,p^-)
               ~\geq~ \left[ (r + \lambda) - \p{\dfrac{\lambda +r}{p^{-}(x)}  + \sigma^2 } \gamma\right] kx^\gamma \\
               &\geq~ \left[ (r + \lambda) - \p{\dfrac{\lambda +r}{p^{-}(\bar{x}_0)}  + \sigma^2 } \gamma\right] kx^\gamma~=~\left[ (r + \lambda) - \p{\dfrac{\lambda +r}{\theta_{\min}}  + \sigma^2 } \gamma\right] kx^\gamma ~\geq~0
    \end{align*}
    for all $x \in [0,\bar{x}_0]$. In turn, we have
        \begin{equation}
        1 - kx^\gamma ~\leq~ p_\ve(x) ~\leq~ 1 \quad \forall x \in[0,\bar{x}_0] \,,
    \end{equation}
    and this yields \eqref{px0}.

{\bf Step 3.} To complete the proof, we prove that
\begin{equation}
        \label{pxs} \lim_{x \to x^{*-}} p(x) ~=~ \theta(x^*).
\end{equation}
Given $\ve \in \left] 0,1/2\right[$ sufficiently small, we will provide an upper bound on the ${\bf L}^{\infty}$ distance of $p_{\ve}$ and $\theta$ over $[x^*-\delta,x^*$] denoted by
\[
I_\delta ~:=~  \max_{x \in [x^* - \delta, x^*]} \big| p_\ve(x) - \theta(x) \big|\qquad\forall \delta\in [0, x^*/4].
\]
For a fixed  $0<\delta<\ds {x^*\over 4}$, the continuity of $p_\ve$ and $\theta$ implies that
\begin{equation}\label{I-delta-1}
I_{\delta}~=~\sg{\theta(x_{m})-p_\ve(x_{m})}\cdot\left[\theta(x_{m})-p_\ve(x_{m})\right],
\end{equation}
for some $x_m\in  [x^* - \delta, x^*]$.  Assume that $p_\ve(x_{m}) \neq \theta(x_{m})$. Since $p_{\ve}(0)=\theta(0)=1$ and $p_{\ve}(x^*)=\theta(x^*)$,  we can define  two points
\[
0~\leq~x_{1}^* ~:=~ \max \sett{ x \in [0,x_{m}[ }{ p_\ve(x) = \theta(x) } ~<~x_{m}
\]
and
\[
{3x^*\over 4}~\leq~x_{m}~<~x_{2}^* ~:=~ \min \sett{ x \in [x_{m},x^*] }{ p_\ve(x) = \theta(x) }~\leq~x^*.
\]
Notice that $p_{\ve}-\theta$ does not change sign in the interval $]x_{1}^*,x_{2}^*[$. For simplicity, let us introduce the following function
\[
q_{\ve}(x)~:=~\sg{\theta(x_{m})-p_\ve(x_{m})}\cdot p_{\ve}(x).
\]
It is clear that
\[
|q'_{\ve}(x)|~=~|p'_{\ve}(x)|\qquad\forall x\in ]x_{1}^*,x_{2}^*[.
\]
Thus, by the second equation in \eqref{sysel} we estimate
\begin{align*}
q''_{\ve}(x)~=&~\sg{\theta(x_{m})-p_\ve(x_{m})}\cdot p''_{\ve}(x)\\
\leq&~ \p{  \dfrac{2}{2\ve + \sigma^2 x^2 } } \cdot \left[r+\lambda+v_{\max}+\big |H_\xi\cdot p_\ve'(x)\big|- \rho_\ve(x) \cdot \big| p_\ve(x) - \theta(x)\big|\right]\\
~\leq&~ \p{  \dfrac{2}{2\ve + \sigma^2 x^2 } } \cdot \left [r+\lambda+v_{\max}+\big |H_\xi\cdot q_\ve'(x)\big |\right]\quad\forall x \in ]x_{1}^*,x_{2}^*[\,.
\end{align*}
Set $\bar{x}_1^*:=\max\left\{x_1^*,\ds {x^*\over 2}\right\}$, we then have
\begin{equation}\label{q''-est}
q''_{\ve}(x)~\leq~{8\over (\sigma x^*)^2}\cdot  \left(r+\lambda+v_{\max}+ K_1\cdot |q'_\ve(x)|\right)\quad\forall x \in ]\bar{x}_{1}^*,x_{2}^*[,
\end{equation}
where $K_1$ is defined in \eqref{k1}. Two cases may occur:
\medskip \\
 CASE 1: If $x_m-\bar{x}_1^*>\delta$  then
\begin{eqnarray*}
q_{\ve}(x_m)-q_{\ve}(x_m-\delta)&=&\sg{\theta(x_{m})-p_\ve(x_{m})}\cdot \left[p_{\ve}(x_m)-p_{\ve}(x_m-\delta)\right]\cr\cr
&\leq&\sg{\theta(x_{m})-p_\ve(x_{m})}\cdot \left[\theta(x_m)-\theta(x_m-\delta)\right]~\leq~ \Delta_{\delta}\theta\cdot \delta,
\end{eqnarray*}
where $\Delta_{\delta}\theta=\sup_{x\in [0,x^*-\delta]} \left|{\theta(x+\delta)-\theta(x)\over \delta}\right|$. Thus, by mean value theorem there exists ${\bar x \in \big] x_m-\delta,x_m\big[}$ such that
\[
q'_{\ve}(\bar{x})~=~{q_{\ve}(x_m)-q_{\ve}(x_m-\delta)\over \delta}~\leq~\Delta_{\delta}\theta.
\]
Let $g$ be solution to the ODE
\[
g'(x)~=~{8\over (\sigma x^*)^2}\cdot (r + \lambda + v_{\max}+K_1\cdot g(x)),\qquad g(\bar x)~=~\Delta_{\delta}\theta~\geq~q'_{\ve}(\bar{x}).
\]
Solving the above equation, one gets
\[
g(x)~=~\p{ {r + \lambda + v_{\max}\over K_1} +\Delta_{\delta}\theta }\cdot \exp\left({8K_1\over(\sigma x^*)^2}\cdot (x - \bar x)\right) - {r + \lambda +  \vmax \over K_1}~\geq ~0,
\]
for all $x\geq \bar{x}$. In particular, it holds that
\[
g'(x)~=~{8\over (\sigma x^*)^2}\cdot \left(r + \lambda + v_{\max}+K_1\cdot |g(x)|\right)\qquad\forall x\in \left]\bar x,x_{2}^*\right[.
\]
A standard comparison argument yields
\[
q'_{\ve}(x)~\leq~g(x)~\leq~\p{ {r + \lambda + v_{\max}\over K_1} + \Delta_{\delta}\theta }\cdot  \exp\left({ 4K_1\over \sigma^2 x^*}\right)\quad\forall x \in ] \bar x, x_2^*[.
\]
Thus, from (\ref{I-delta-1}) it holds
\begin{align}
I_{\delta}~&=~\sg{ \theta(x_m) - p_\ve(x_m)}\cdot \p{ \theta(x_m) - \theta(x_2^*)} + q_\ve(x_2^*) - q_\ve(x_m) \nonumber \\
~&\leq~\sup_{x,y\in [x^*-\delta,x^*]}|\theta(x)-\theta(y)| + \p{ {r + \lambda + v_{\max}\over K_1} + \Delta_{\delta}\theta }\cdot \exp\left({ 4K_1\over \sigma^2 x^*}\right)\cdot \delta.  \nonumber \\
~&\leq~{r + \lambda + v_{\max}\over K_1}\cdot \exp\left({ 4K_1\over \sigma^2 x^*}\right)\cdot\delta+\left[\exp\left({ 4K_1\over \sigma^2 x^*}\right) + 1\right]\cdot \sup_{x,y\in [x^*-\delta,x^*]}|\theta(x)-\theta(y)|. \label{estI1}
\end{align}

CASE 2: Let us assume that   $0<x_m-\bar{x}_1^*\leq \delta$. Since $\delta <\ds {x^*\over 4}$ and $x_m\geq\ds {3x^*\over 4} $, we have that $\bar{x}^*_1>\ds {x^*\over 2}$ and this yields $\bar{x}^*_1=x_1^*$. Two subcases are considered:
\begin{itemize}
\item If $q_\ve(x_{m})-q_{\ve}(x^*_1)\geq 0$ then (\ref{I-delta-1}) implies that
\begin{align}\label{estI2}
I_{\delta}~=&~\sg{\theta(x_{m})-p_\ve(x_{m})}\cdot [\theta(x_m)-\theta(x_1^*)]+q_{\ve}(x_1^*)-q_{\ve}(x_m) \nonumber \\[.8ex]
~\leq&~ |\theta(x_m)-\theta(x_1^*)|~\leq\sup_{x,y\in [x^*-\delta,x^*]}|\theta(x)-\theta(y)|.
\end{align}
\item Otherwise, if $q_\ve(x_{m})-q_{\ve}(x^*_1)< 0$ then by mean value theorem  there exists $\tilde{x} \in \big] x_1^*,x_m\big[$ such that
\[
q'_{\ve}(\tilde{x})~=~{q_{\ve}(x_m)-q_{\ve}(x_1^*)\over x_m-x_1^*}~<~0.
\]
With the same argument in case 1, one can show that
\[
q'_{\ve}(x)~\leq~\tilde{g}(x)~\leq~ {r + \lambda + v_{\max}\over K_1} \cdot  \exp\left({ 4K_1\over \sigma^2 x^*}\right)\quad\forall x \in ] \tilde{x}, x_2^*[\,,
\]
where $\tilde{g}$ is the solution to
\[
\tilde{g}'(x)~=~{8\over (\sigma x^*)^2}\cdot (r + \lambda + v_{\max}+K_1\cdot \tilde{g}(x)),\qquad \tilde{g}(\tilde{x})~=~0~\geq~q'_{\ve}(\tilde{x}).
\]
Thus, as in (\ref{estI1}), it holds
\begin{equation}\label{estI3}
I_{\delta}~\leq~{r + \lambda + v_{\max}\over K_1}\cdot \exp\left({ 4K_1\over \sigma^2 x^*}\right)\cdot\delta+\sup_{x,y\in [x^*-\delta,x^*]}|\theta(x)-\theta(y)|.
\end{equation}
\end{itemize}
From (\ref{estI1})-(\ref{estI3}), we obtain that
\[
\left\|p_{\ve}-\theta\right\|_{{\bf L}^{\infty}([x^*-\delta,x^*])}~\leq~C_1\cdot \delta+C_2\cdot \sup_{x,y\in [x^*-\delta,x^*]}|\theta(x)-\theta(y)|,
\]
for all $\ve \in \left] 0,\frac{1}{2}\right[$, $\delta\in \left]0,{x^*\over 4}\right[$ with the constants
\[
C_1~=~{r + \lambda + v_{\max}\over K_1}\cdot \exp\left({ 4K_1\over \sigma^2 x^*}\right),\qquad C_2~=~\exp\left({ 4K_1\over \sigma^2 x^*}\right) + 1.
\]
In particular,
\[
\|p-\theta\|_{{\bf L}^{\infty}(]x^*-\delta,x^*[)}~\leq~C_1\cdot \delta+C_2\cdot \sup_{x,y\in [x^*-\delta,x^*]}|\theta(x)-\theta(y)|
\]
and the uniform continuity of $\theta$ yields (\ref{pxs}).
\end{proof}
\subsection{Optimal currency devaluation and payment strategy near  $x^*$}
This subsection is devoted to the behavior of  optimal feedback controls  near  the bankruptcy threshold $x^*$. Roughly speaking, let $(u^*,v^*,p)$ be an optimal solution to the problem of optimal debt management (\ref{SDE})--(\ref{C1}), and $V$ be the corresponding value function. In addition to (\ref{vas}), we will show that when the debt-to-income ratio $x$ closes to $x^*$ {\it
\begin{itemize}
\item if the risk of bankruptcy $\rho$ slowly approaches to infinity then the optimal strategy of borrower involves continuously devaluating its currency and making payment, i.e. $u^*(x)>0$ and $v^*(x)>0$,
\item conversely, if the risk of bankruptcy $\rho$ quickly approaches to infinity  then any action to reduce the debt is not optimal, i.e. $u^*(x)=v^*(x)=0$.
\end{itemize}
}
Let us first establish upper and lower bounds for $V$. Recalling that
\[
p~\in~ [\theta_{\min},1]\qquad\mathrm{with}\qquad \theta_{\min}~=~\min\left\{\theta(x^*), \dfrac{r+\lambda}{r+\lambda+\vmax}\right\},
\]
we introduce a non-decreasing function $\beta: [0,x^*[\to [0,+\infty[$ defined by
\begin{equation}\label{beta}
\beta(t)~=~\max_{s\in [0,t]}\left[\rho(s)\ln\left(t\over s\right)\right]+{\lambda + r\over \theta_{\min}}+{\sigma^2\over 2}~<~+\infty\qquad\forall t\in [0,x^*[.
\end{equation}
\v

\begin{proposition}\label{BB}
Under the same assumptions in Theorem \ref{ext}, it holds
\begin{equation}\label{bbound11}
V(x)~\leq~V_1(x)~:=~B\cdot \inf_{t\in [x,x^*[}\left[{\beta(t)\over r\ln\left({t\over x}\right)+\beta(t)}\right]\qquad\forall x\in ]0,x^*[.
\end{equation}
In addition if we assume that
\begin{equation}\label{asp1}
x^*~\geq~ {2\over r+\lambda}\qquad\mathrm{and}\qquad \lim_{x\to x^*-}\rho(x) (x^*-x)^2~=~+\infty,
\end{equation}
then there exists $x^{\diamond}\in [x^*/2,x^*[$ such that
\begin{equation}\label{lboud11}
V(x)~\geq~V_2(x)~:=~B\cdot \left(1-{\ln\p{ \dfrac{x^*}{x}}\left(\ds 1-{x\over x^*}\right)\over \ln\p{ \dfrac{x^*}{x^{\diamond}}}(x^*-x^\diamond)}\right)\qquad\forall x\in \left[x^{\diamond}, x^*\right].
\end{equation}
\end{proposition}
\begin{proof}
{\bf 1.} For any given $x_1\in ]0,x^*[$ and $x_2\in ]0,x_1[$, we seek for an upper bound for  $V$ as  a supersolution  to the first equation of (\ref{sys}) of the form
\[
V_1(x)~=~\begin{cases}
B\qquad &\mathrm{if}\qquad x \in [x_1,x^*],\cr
B \left(1- \alpha \ln\p{ \dfrac{x_1}{x}}\right)\qquad &\mathrm{if}\qquad x\in [x_2,x_1],\\[1.4ex]
B \left(1- \alpha \ln\p{ \dfrac{x_1}{x_2}}\right)\qquad &\mathrm{if}\qquad x\in [0,x_2],
\end{cases}
\]
with $\alpha>0$ satisfying the following relation
\[
-r+\alpha\cdot \left(r\ln\left({x_1\over x_2}\right)+\beta(x_1)\right)~=~0.
\]
It is clear that
\[
V_1(0)~=~B \left(1- \alpha \ln\p{ \dfrac{x_1}{x_2}}\right)~\geq~0~=~V(0)\qquad\mathrm{and}\qquad V_1(x^*)~=~B~=~V(x^*).
\]
For every $x\in ]0,x_2[$, it holds
\begin{align*}\label{sbest1}
&-  (r+\rho(x))V_1(x)+\rho(x) B+ H(x, V_1'(x),p(x)) + \dfrac{\sigma^2x^2}{2}\cdot V_1''(x)\nonumber\\
~&=~- (r+\rho(x)) V_1(x)+\rho(x) B ~\leq~B\cdot \left[-r+\alpha (r+\rho(x_2))\ln\left(x_1\over x_2\right)\right]\\
 ~&\leq~B\cdot \left(-r+\alpha\cdot \left(\beta(x_1)+r\ln\left({x_1\over x_2}\right)\right)\right)~=~0\nonumber.
\end{align*}
Similarly, for every $x\in ]x_1,x^*[$, one has
\begin{multline*}
-  (r+\rho(x)) V_1(x)+\rho(x) B+ H(x, V_1'(x),p(x)) + \dfrac{\sigma^2x^2}{2}\cdot V_1''(x)\\
~=~- (r+\rho(x))V_1(x)+\rho(x) B~=~-rB~<~0.
\end{multline*}
On the other hand, for every $x\in]x_2,x_1[$,  we compute
\[
V_1'(x)~=~{B\alpha\over x}~>~0\qquad\mathrm{and}\qquad V_1''(x)~=~-{B\alpha\over x^2},
\]
and use Lemma  \ref{H-C1} to obtain
\begin{align*}
-  (r+\rho(x)) &V_1(x)+\rho(x) B+ H(x, V_1'(x),p(x)) + \dfrac{\sigma^2x^2}{2}\cdot V_1''(x)\\
~\leq~&- (r+\rho(x)) V_1(x)+\rho(x) B+ \left({\lambda+r\over \theta_{\min}}+\sigma^2\right)xV'_1(x)+ \dfrac{\sigma^2x^2}{2}\cdot V_1''(x)\\
~=~&B\cdot \left[-r+\alpha\cdot \left((r+\rho(x))\ln\left({x_1\over x}\right)+{\lambda+r\over \theta_{\min}}+{\sigma^2\over 2}\right)\right]\\
~\leq~&B\cdot \left[-r+\alpha\cdot \left(r\ln\left({x_1\over x_2}\right)+\beta(x_1)\right)\right]~=~0.
\end{align*}
Hence, $V_1$ is a supersolution of  the first equation of (\ref{sys})  and
\[
V(x)~\leq~V_1(x),\qquad\forall x\in [0,x^*].
\]
In particular, we have
\[
V(x_2)~\leq~V_1(x_2)~=~B\cdot\left(1-{r\over r\ln\left({x_1\over x_2}\right)+\beta(x_1)}\cdot  \ln\p{ \dfrac{x_1}{x_2}}\right),
\]
and this implies that
\[
V(x_2)~\leq~~B\cdot {\beta(x_1)\over r\ln\left({x_1\over x_2}\right)+\beta(x_1)}.
\]
Since the above inequality hold for every $x_2\in ]0,x_1[$, one obtains (\ref{bbound11}).
\medskip

{\bf 2.} We  seek for  a lower bound for  $V$ as  a subsolution  to the first equation of (\ref{sys}) in the interval $[x_1,x^*]$ of the form
\[
V_2(x)~=~B \left(1- \alpha_1 \ln\p{ \dfrac{x^*}{x}}(x^*-x)\right),\qquad\forall x\in [x_1,x^*],
\]
with
\begin{equation}\label{alpha1}
x_1~\in~ [x^*/2, x^*]\qquad\mathrm{and}\qquad \alpha_1~=~\left[\ln\p{ \dfrac{x^*}{x_1}}(x^*-x_1)\right]^{-1}~\geq~ \ds {2\over \ln2\cdot x^*},
\end{equation}
such that $V_2(x_1)=0$. For every $x\in ]x_1,x^*[$, we compute
\begin{equation}\label{V2''}
V'_2(x)~=~B\alpha_1\cdot\left(\ln\p{ \dfrac{x^*}{x}}+{x^*-x\over x}\right)~\geq~0 \quad\mathrm{and}\quad V_2''(x)~=~-B\alpha_1\left({1\over x}+{x^*\over x^2}\right).
\end{equation}
For simplicity, set
\[
C_1~:=~\lambda+\mu + \vmax\,.
\]
Using Lemma \ref{H-C1} and the first condition of (\ref{asp1}), we have
\[
H(x, V_2'(x),p(x))~\geq~-C_1V_2'(x)x,\qquad\forall x\in [x_1,x^*]
\]
which implies that
\begin{align*}
- (r+\rho(x))V_2(x)+\rho(x) B+ H(x, V_2'(x),p(x)) + \dfrac{\sigma^2x^2}{2}\cdot V_2''(x)\\
~\geq~ - (r+\rho(x)) V_2(x)+\rho(x) B-C_1xV'_2(x)+ \dfrac{\sigma^2x^2}{2}\cdot V_2''(x).
\end{align*}
On the other hand, from (\ref{V2''}), one has
\[
C_1xV'_2(x)~\leq~2B\alpha_1C_1(x^*-x)~\leq~B\alpha_1C_1x^*,\qquad \dfrac{\sigma^2x^2}{2}\cdot V_2''(x)~\geq~-B\alpha_1\sigma^2 x^*,
\]
for all $x\in [x_1,x^*]$. Thus,
\begin{align}\label{Rossa}
- (r+\rho(x))V_2(x)+\rho(x) B+ H(x, V_2'(x),p(x)) + \dfrac{\sigma^2x^2}{2}\cdot V_2''(x) \nonumber \\
~\geq~B\cdot \left[\alpha_1\left(\rho(x)\ln\p{ \dfrac{x^*}{x}}(x^*-x)- (C_1+\sigma^2)x^*\right)-r\right].
\end{align}
From (\ref{asp1}), it holds
\[
\lim_{x\to x^*}\rho(x)\ln\p{ \dfrac{x^*}{x}}(x^*-x)~=~+\infty,
\]
and therefore there exists $x^{\diamond}\in ]x^*/2,x^*[$ sufficiently close to $x^*$ such that
\[
{2\over \ln2 \cdot x^*}\cdot \left(\rho(x)\ln\p{ \dfrac{x^*}{x}}(x^*-x)- (C_1+\sigma^2)x^*\right)-r~\geq~0\qquad\forall x\in [x^{\diamond}, x^*[.
\]
In particular, if we choose $x_1=x^{\diamond}$  then (\ref{Rossa}) and (\ref{alpha1}) imply that
\[
-  (r+\rho(x))V_2(x)+\rho(x) B+ H(x, V_2'(x),p(x)) + \dfrac{\sigma^2x^2}{2}\cdot V_2''(x)~\geq~0,
\]
for all $x\in  [x^{\diamond}, x^*[$. Since
\[
V_2(x^{\diamond})~=~0~\leq~V(x^{\diamond})\quad\mathrm{and}\qquad V_2(x^*)~=~B~\leq~V(x^*),
\]
 $V_2$ is a subsolution  to the first equation of (\ref{sys}) in $[x^{\diamond},x^*]$ and thus
\[
V(x)~\geq~V_2(x),\qquad\forall x \in [x^{\diamond},x^*],
\]
which yields (\ref{lboud11}).
%
\end{proof}
\medskip

From the formula of $\beta$ in (\ref{beta}), one can actually show that $\beta$ is locally Lipschitz and
\begin{equation}\label{cond121}
0~\leq~\dot{\beta}(t)~\leq~{\rho(t)\over t}\qquad a.e.~ t\in [x,x^*].
\end{equation}
Notice that $\lim_{t\to 0+}\ds{\rho(t)\over t}=\rho'(0)<+\infty$, if $\ds\int_{x^*-\delta}^{x^*}\rho(t)dt<+\infty$ for some $\delta>0$ then $\ds \int_{0}^{x^*}{\rho(t)\over t}dt<+\infty$. In this case, we have
\begin{equation}\label{beta*}
\sup_{t\in [0,x^*[}\beta(t)~\leq~\beta^*:=\int_{0}^{x^*}{\rho(t)\over t}dt+{\lambda + r\over \theta_{\min}}+{\sigma^2\over 2}.
\end{equation}
%
%
As a consequence of Proposition \ref{BB}, the followings hold.
\begin{corollary} Under the same assumptions in Theorem \ref{ext}, two cases may occur
\begin{itemize}
\item [(i)] Let $\beta^*$ be as in \eqref{beta*}.If $\ds \int_{x^*-\delta}^{x^*}\rho(t)~dt<+\infty$ for some $\delta>0$ and 
$$
\beta^*<Br\cdot \min\left\{\ds{1\over c'(0)}, {1\over L'(0)x^*}\right\},
$$
then there exists some $\bar{x}\in ]0,x^*[$ sufficiently close to $x^*$ such that
\[
u^*(x)~>~0\qquad\mathrm{and}\qquad v^*(x)~>~0\qquad\forall x\in [\bar{x},x^*[.
\]
\item [(ii)] If (\ref{asp1}) holds then there exists $\hat{x}\in ]0,x^*[$ sufficiently close to $x^*$ such that
\[
u^*(x)~=~v^*(x)~=~0\qquad\forall x\in [\hat{x},x^*[.
\]
\end{itemize}
\end{corollary}
\begin{proof} {\bf (i).} For every given $x_2\in ]0,x^*[$, \eqref{bbound11} and (\ref{beta*}) imply that
\[
V(x^*)-V(x_2)~\geq~V(x^*)-V_1(x_2)~\geq~B\cdot {r\ln(x_1/x_2)\over r\ln(x_1/x_2)+\beta^*}\quad\forall x_1\in [x_2,x^*[.
\]
In particular, we have
\[
V(x^*)-V(x_2)~\geq~\sup_{x_1\in [x_2,x^*[}\left[B\cdot {r\ln(x_1/x_2)\over r\ln(x_1/x_2)+\beta^*}\right]~=~B\cdot {r\ln(x^*/x_2)\over r\ln(x^*/x_2)+\beta^*}.
\]
By mean value theorem, there exists $x_c\in [x_2,x^*]$ such that
\begin{equation}\label{Vc}
V'(x_c)\cdot x_c~\geq~B\cdot {r\ln(x^*/x_2)\over r\ln(x^*/x_2)+\beta^*}\cdot {x_2\over x^*-x_2}.
\end{equation}
On the other hand, from the first equation of (\ref{sys}) and Lemma \ref{H-C1}, it holds
\[
{\sigma^2\over 2}[x^2V''(x)+xV'(x)]~\geq~-cxV'(x)-\rho(x)B\qquad\forall x\in ]0, x^*[,
\]
with constant $c={r+\lambda\over \theta_{\min}}+{\sigma^2\over 2}$. Set $Z(x)=xV'(x)$, $c_1={4c\over \sigma^2 x^*}$ and $c_2={4B\over \sigma^2x^*}$, we have
\[
Z'(x)~\geq~-c_1 Z(x)-c_2 \rho(x)\qquad\forall x\in ]x^*/2,x^*[.
\]
Solving the differential inequality yields
\[
Z(x)~\geq~e^{c_1(x_0-x)}\cdot Z(x_0)-c_2\int_{x_0}^{x}\rho(s)ds\quad\forall~~{x^*\over 2}<x_0\leq x<x^*.
\]
In particular, recalling (\ref{Vc}), we have
\[
Z(x)~\geq~Be^{c_1(x_2-x^*)}{r\ln(x^*/x_2)\over r\ln(x^*/x_2)+\beta^*}\cdot {x_2\over x^*-x_2}-c_2\int_{x_2}^{x^*}\rho(s)ds~=:~I(x_2),
\]
for all $x^*/2<x_2<x_c\leq x<x^*$. Since $\beta^*<\ds Br\cdot \min\left\{\ds{1\over c'(0)}, {1\over L'(0)x^*}\right\}$ and $\ds\int_{0}^{x^*}\rho(t)~dt<+\infty$, it holds
\[
\lim_{x_2\to x^*}I(x_2)~=~{Br\over \beta^*}~>~\max\left\{c'(0), {L'(0) x^*}\right\}.
\]
Thus, there exists $x_2\in [x^*/2,x^*[$ sufficiently close to $x^*$ such that
\[
Z(x)~\geq~I(x_2)~>~c'(0)\qquad\mathrm{and}\qquad V'(x)~\geq~{I(x_2)\over x}~>~L'(0)~\geq~L'(0)p(x),
\]
for all $x\in [x_c,x^*]$ and, recalling \eqref{ufb}-\eqref{vfb}, this yields (i) for $\bar{x}=x_c$.
\medskip

{\bf (ii).} Assuming that (\ref{asp1}) holds, we recall $x^{\diamond}$ in Proposition \ref{BB}. For every $x_1\in \left[x^{\diamond}, x^*\right[$, it holds
\begin{equation}\label{Stev}
V'(x_c)\cdot x_c~=~{V(x^*)-V(x_1)\over x^*-x_1}\cdot x_c~\leq~{B-V_2(x_1)\over x^*-x_1 }\cdot x^*
~=~{B\over  \ln\p{ \dfrac{x^*}{x^{\diamond}}}(x^*-x^\diamond)}\cdot \ln\left({x^*\over x_1}\right),
\end{equation}
for some $x_c\in ]x_1,x^*[$. From the first equation of (\ref{sys}) and Lemma \ref{H-C1}, one estimates
\[
{\sigma^2\over 2}[x^2V''(x)+xV'(x)]~\leq~rB+\left(\lambda+\mu+v_{\max}\right)xV'(x)\qquad\forall x\in ]x^*/2, x^*[.
\]
Recalling that $Z(x)=xV'(x)$, we have
\[
Z'(x)~\leq~c_3 Z(x)+c_4\qquad\forall x\in [x^*/2,x^*],
\]
with $c_3:={4rB\over \sigma^2 x^*}$ and $c_4:={2\left(\lambda+\mu+v_{\max}\right)\over \sigma^2 x^*}$. Thus,
\[
Z(x)~\leq~e^{c_3\cdot (x-x_c)}\cdot Z(x_c)+{c_4\over c_3}\cdot \left(e^{c_3\cdot (x-x_c)}-1\right)
\]
and (\ref{Stev}) implies that
\[
Z(x)~\leq~{B\over  \ln\p{ \dfrac{x^*}{x^{\diamond}}}(x^*-x^\diamond)}\cdot e^{c_3\cdot (x^*-x_1)}\cdot \ln\left({x^*\over x_1}\right)+{c_4\over c_3}\cdot \left(e^{c_3\cdot (x^*-x_1)}-1\right) ~=:~J(x_1),
\]
for all $x\in [x_c,x^*[$. Since $\lim_{x_1\to x^*-}J(x_1)=0$, there exists $x_1\in ]x^\diamond,x^*[$ such that
\[
Z(x)~\leq~J(x_1)~\leq~c'(0)\quad\mathrm{and}\quad V'(x)~\leq~{J(x_1)\over x}~\leq~\theta_{\min}L'(0)~\leq~L'(0)\cdot p(x),
\]
for all $x\in [x_c,x^*[$ and, recalling \eqref{ufb}-\eqref{vfb}, setting $\hat{x} = x_c$ yields (ii).
\end{proof}
\v

{\bf Acknowledgments.} The research by K. T. Nguyen was partially supported by a grant from
the Simons Foundation/SFARI (521811, NTK).

\normalem

\end{document}